\documentclass{amsproc}

\usepackage{graphicx,float}
\usepackage{amsfonts,amsmath,latexsym,amssymb,amsthm}

\DeclareMathOperator{\Aut}{\mathrm{Aut}}
\DeclareMathOperator{\Li}{\mathrm{Li}}
\DeclareMathOperator{\Norm}{\mathrm{Norm}}

 \newtheorem{theorem}{Theorem}[section]
 \newtheorem{lemma}[theorem]{Lemma}
 
 \newtheorem{question}[theorem]{Question}

 \theoremstyle{definition}

 \theoremstyle{remark}
 
 \newcommand{\mc}{\mathcal}

 \newcommand{\D}{\mc{D}}

 \newcommand{\pP}{\mathfrak{P}}
 \newcommand{\pp}{\mathfrak{p}}

 \newcommand{\A}{\mathbb{A}}
 \newcommand{\C}{\mathbb{C}}
 
 \newcommand{\R}{\mathbb{R}}
 
 \newcommand{\Q}{\mathbb{Q}}
 \newcommand{\oQ}{\overline{\Q}}
 \newcommand{\Z}{\mathbb{Z}}

 \newcommand{\bbalpha}{\overline{\alpha}}

 \newcommand{\h}{\frac12}

 \newcommand{\dt}{\text{\rm d}t}

\begin{document}

\title{A note on generators of number fields}

\author{Jeffrey D. Vaaler}
\address{Department of Mathematics, University of Texas at Austin, 1 University Station C1200, Austin, Texas 78712}
\email{vaaler@math.utexas.edu}
\thanks{The first author was supported in part by NSA Grant \#H92380-12-1-0254.}

\author{Martin Widmer}
\address{Department for Analysis and Computational Number Theory,
Graz University of Technology, 8010 Graz, Austria}
\email{widmer@math.tugraz.at}
\thanks{The second author was supported by an FWF grant \#M1222-N13.}

\date{\today, and in revised form ....}

\dedicatory{}


\begin{abstract}
We establish upper bounds for the smallest height of a generator of a number field $k$ over the rational field $\Q$. 
Our first bound applies to all number fields $k$ having at least one real embedding.  We also give a second conditional result for all
number fields $k$ such that the Dedekind zeta-function associated to the Galois closure of $k/\Q$ satisfies GRH.  This provides
a partial answer to a question of W. Ruppert.
\end{abstract}

\maketitle



\numberwithin{equation}{section}

\section{Introduction}

Let $\oQ$ denote an algebraic closure of the field $\Q$ of rational numbers, and for $\alpha$ in $\oQ$ let
$H(\alpha)$ denote the absolute, multiplicative Weil height of $\alpha$.  If $k \subseteq \oQ$ is an algebraic number field, we
consider the problem of showing that $k = \Q(\alpha)$, where the height of $\alpha$ can be estimated by invariants of $k$.  In
particular we are interested in showing that $k$ is generated over $\Q$ by an element $\alpha$
that has relatively small height.  

Let $\Delta_k$ denote the discriminant of the number field $k$.  In \cite[Question 2]{Rupp} W.~Ruppert proposed the following more 
precise question.

\begin{question}{\sc [Ruppert, 1998]}\label{con1}  Does there exist a positive constant $B = B(d)$ such that, if $k$ is an algebraic 
number field of degree $d$ over $\Q$, then there exists an element $\alpha$ in $k$ such that
\begin{equation}\label{gen2}
k = \Q(\alpha),\quad\text{and}\quad H(\alpha) \le B \bigl|\Delta_k\bigr|^{1/2d}?
\end{equation} 
\end{question}

In fact Ruppert stated this question using the naive height of $\alpha$, but elementary inequalities between the two heights imply 
that (\ref{gen2}) is equivalent to the bound that Ruppert proposed.  Ruppert himself gave a positive answer to Question \ref{con1}
for $d=2$ (\cite[Proposition 2]{Rupp}), and also for totally real number fields $k$ of prime degree (\cite[Proposition 3]{Rupp}).  
The analogous question for function fields in positive characteristic has been answered positively by the second author in \cite{art5}.
In this note we give a positive answer for all number fields $k$ having at least one embedding into $\R$, and with a constant 
$B \le 1$ for all such $k$.  Our argument is a simple application of Minkowski's first theorem in the geometry of numbers over the 
adele space $k_{\A}$ associated to $k$. 

\begin{theorem}\label{thm1}  Let $k$ be an algebraic number field with degree $d$ over $\Q$, and assume that $k$ has 
at least one embedding into $\R$.  Then there exists an element $\alpha$ in $k$, such that
\begin{equation}\label{gen20}
k = \Q(\alpha),\quad\text{and}\quad H(\alpha) \le \Bigl(\frac{2}{\pi}\Bigr)^{s/d} \bigl|\Delta_k\bigr|^{1/2d},
\end{equation} 
where $s$ is the number of complex places of $k$.
\end{theorem}

If all the embeddings of $k$ into $\C$ are complex, then the situation is more complicated.  In this case an alternative application of 
Minkowski's first theorem over $k_{\A}$ leads to a bound that depends on the existence of certain rational prime numbers $p$
such that the principal ideal $\langle p \rangle$ in $O_k$ has a prime ideal factor
with residue class degree equal to $1$.  It is also necessary that a finite product of such rational primes be slightly larger than $|\Delta_k|^{\h}$.
In order to establish the existence of such rational primes, we must assume that the Dedekind zeta-function 
$\zeta_l(s)$ associated with the Galois closure $l$ of the extension $k/\Q$ satisfies the Generalized Riemann Hypothesis. This 
leads to a bound for the height of a generator as anticipated by Ruppert in (\ref{gen2}), and 
with $B = B(d)$ depending effectively on the degree $d$.

\begin{theorem}\label{thm2}  For each integer $d \ge 2$ there exists an effectively computable positive number $B = B(d)$ having 
the following property.  Let $k$ be an algebraic number field with degree $d$ over $\Q$, let $l$ be the Galois closure of 
the extension $k/\Q$, and assume that the Dedekind zeta-function $\zeta_l(s)$ associated to $l$ satisfies the 
Generalized Riemann Hypothesis.  Then there exists an element $\alpha$ in $k$, such that
\begin{equation}\label{gen90}
k = \Q(\alpha),\quad\text{and}\quad H(\alpha) \le B \bigl|\Delta_k\bigr|^{1/2d}.
\end{equation} 
\end{theorem}

 We obtain this conditional result by applying an effective version of  the Chebotarev density theorem proved by Lagarias and 
Odlyzko in \cite{lo1977}.

\section{Preliminaries}
We suppose that $r$ is the number of real places of $k$, and $s$ is the number of 
complex places of $k$, so that $r + 2s = d$.   Then we define the field constant
\begin{equation}\label{gen0}
c_k = \Bigl(\frac{2}{\pi}\Bigr)^{s/d} \bigl|\Delta_k\bigr|^{1/2d}.
\end{equation}
We have
\begin{equation*}\label{gen1}
\Bigl(\frac{2}{\pi}\Bigr)^{\h} \le \Bigl(\frac{2}{\pi}\Bigr)^{s/d} \le1,
\end{equation*}
so that this factor is of no significance for our purposes here.  However, the constant $c_k$ occurs naturally
in a basic formula for the Haar measure of certain subsets of the adele ring $k_{\A}$ associated 
to $k$.  At each place $v$ of $k$ we let $k_v$ denote the completion
of $k$ with respect to an absolute value from $v$, and we write $d_v = d_v(k/\Q) = [k_v:\Q_v]$ for the local degree at $v$.  
Let $\|\ \|_v$ be the unique absolute value at the place $v$ which extends either the usual Euclidean absolute value on $\Q$, or the 
usual $p$-adic absolute value on $\Q$.  We also define a second absolute value $|\ |_v$ from the place $v$ by setting
\begin{equation*}\label{gen41}
|\ |_v = \|\ \|_v^{d_v/d}.
\end{equation*}
At each place $v$ we define $O_v\subseteq k_v$ by
\begin{equation*}\label{gen3}
O_v = \begin{cases} \{\xi\in k_v: |\xi|_v < 1\}& \text{if $v|\infty$},\\
                                     \{\xi\in k_v: |\xi|_v \le 1\}& \text{if $v\nmid\infty$}.\end{cases}
\end{equation*}
It follows that 
\begin{equation*}\label{gen4}
\prod_v O_v \subseteq k_{\A},
\end{equation*}
and from the product formula we get
\begin{equation*}\label{gen5}
 k\cap \prod_v O_v = \{0\}.
\end{equation*}

At each place $v$ of $k$ we select a Haar measure $\beta_v$ defined on the Borel subsets of $k_v$ and
normalized as follows.  If $v$ is a real place then $\beta_v$ is the ordinary Lebesgue measure, and if $v$ is a complex
place then $\beta_v$ is the Lebesgue measure on $\C$ multiplied by $2$.  If $v$ is a non-archimedean place then
we require that
\begin{equation*}\label{haar1}
\beta_v(O_v) = |\D_v|_v^{d/2},
\end{equation*}
where $\D_v$ is the local different of $k$ at $v$.  Now let $S$ be a finite subset of places of $k$ containing all the archimedean
places, and let
\begin{equation*}\label{haar3}
k_{\A}(S) = \prod_{v\in S} k_v \times \prod_{v\notin S} O_v
\end{equation*}
be the corresponding open subgroup of the additive group of $k_{\A}$.  We write $\beta$ for the unique Haar measure on the
Borel subsets of $k_{\A}$ such that the restriction of $\beta$ to each open subgroup $k_{\A}(S)$ is the product measure
\begin{equation*}\label{haar4}
\prod_v \beta_v.  
\end{equation*}
The Haar measure $\beta$, normalized in this way, has the property that it induces a Haar measure $\beta^{\prime}$ on 
the Borel subsets of the compact quotient group $k_{\A}/k$ such that
\begin{equation*}\label{haar5}
\beta^{\prime}\bigl(k_{\A}/k\bigr) = 1.
\end{equation*}
Using the basic identity
\begin{equation}\label{haar6}
\prod_{v \nmid \infty} |\D_v|_v^{-d} = |\Delta_k|,
\end{equation}
we find that
\begin{equation}\label{gen6}
\beta\Big\{\prod_v O_v\Big\} = 2^d \Bigl(\frac{\pi}{2}\Bigr)^s \bigl|\Delta_k\bigr|^{-1/2} = 2^d c_k^{-d}.
\end{equation}

Suppose more generally that $\gamma = (\gamma_v)$ is an element of the multiplicative group $k_{\A}^{\times}$ of
ideles associated to $k$.  Then at each place $v$ of $k$ we have
\begin{equation*}\label{gen7}
\gamma_v O_v = \begin{cases} \{\xi\in k_v: |\xi|_v < |\gamma_v|_v\}& \text{if $v|\infty$},\\
                                                          \{\xi\in k_v: |\xi|_v \le |\gamma_v|_v\}& \text{if $v\nmid\infty$}.\end{cases}
\end{equation*}
It follows that
\begin{equation*}\label{gen8}
\prod_v \gamma_v O_v \subseteq k_{\A},
\end{equation*}
and 
\begin{align}\label{gen9}
\begin{split}
\beta\Big\{\prod_v \gamma_v O_v\Big\}  &= \prod_v \beta_v(\gamma_v O_v)\\
                                                                        &= \prod_v \Big\{\|\gamma_v\|_v^{d_v} \beta_v(O_v)\Big\}\\
                                                                        &= 2^d c_k^{-d} \Big\{\prod_v |\gamma_v|_v\Big\}^d.
\end{split}
\end{align}
If 
\begin{equation*}\label{gen10}
\prod_v |\gamma_v|_v \le 1, 
\end{equation*}
then from the product formula we get
\begin{equation*}\label{gen11}
 k\cap \prod_v \gamma_vO_v = \{0\}.
 \end{equation*}
On the other hand, if 
\begin{equation}\label{gen12}
c_k < \prod_v |\gamma_v|_v,
\end{equation}
then there exists a nonzero point $\alpha$ in
\begin{equation*}\label{gen13}
k\cap \prod_v \gamma_v O_v.
\end{equation*}
That is, $\alpha \not= 0$ belongs to $k$ and satisfies the system of inequalities
\begin{equation}\label{gen14}
|\alpha|_v < |\gamma_v|_v\quad\text{if}\quad v | \infty,
\end{equation}
and
\begin{equation}\label{gen15}
|\alpha|_v \le |\gamma_v|_v\quad\text{if}\quad v \nmid\infty.
\end{equation}
The existence of $\alpha$ follows immediately from the adelic version of Minkowski's first theorem, (see \cite[Theorem  3]{BV}
for a more detailed account of geometry of numbers over adeles spaces.)

\section{Proof of Theorem \ref{thm1}}

If the number field $k$ has an embedding into $\R$, then there exists an archimedean place 
$w$ of $k$ such that $k_w = \R$ and $[k_w: \Q_{\infty}] = 1$.   Let $\rho$ be a real parameter such that $c_k < \rho$.
We select $\gamma = (\gamma_v)$ in $k_{\A}^{\times}$ so that
\begin{equation*}\label{gen21}
|\gamma_v|_v = \begin{cases}  1 & \text{if $v|\infty$ and $v \not= w$,}\\
                                                      \rho & \text{if $v = w$,}\\
                                                         1 & \text{if $v \nmid \infty$}.\end{cases} 
\end{equation*} 
Because $c_k < \rho$ we have
\begin{equation*}\label{gen22}
c_k < \rho = \prod_v |\gamma_v|_v,
\end{equation*}
and this verifies (\ref{gen12}).  Hence there exists a point $\alpha \not= 0$ in
\begin{equation}\label{gen23}
k\cap \prod_v \gamma_v O_v.
\end{equation}

If $w$ is the {\it only} archimedean place of $k$ then 
\begin{equation*}\label{gen24}
[k : \Q] = [k_w : \Q_{\infty}] = 1,
\end{equation*} 
and the statement of the theorem is trivial.  Hence we may assume that $k$ has at least two archimedean 
places.  Then (\ref{gen14}) and (\ref{gen15}) imply that $\alpha$ satisfies the inequalities
\begin{equation}\label{gen25}
1 < |\alpha|_w < \rho,\quad\text{and}\quad H(\alpha) \le \prod_v \max\{1, |\gamma_v|_v\} = \rho.
\end{equation}
Let $\Q(\alpha) = k^{\prime} \subseteq k$, and let 
$u$ be an infinite place of $k^{\prime}$ such that $w|u$.  Because $\alpha$ belongs to $k^{\prime}$, the
map $v\mapsto \|\alpha\|_v$ is constant on the set of places $v$ of $k$ such that $v|u$.  It follows from our
choice of $\gamma = (\gamma_v)$ that $w$ is the only place of $k$ that lies over $u$.  Then we have
\begin{equation*}\label{gen26}
[k : k^{\prime}] = [k_w : k_u^{\prime}] \le [k_w : \Q_{\infty}] = 1,
\end{equation*}
and therefore $k = k^{\prime} = \Q(\alpha)$.  The inequality on the right of (\ref{gen25}) shows that
$H(\alpha) \le \rho$ must hold for every positive number $\rho$ such that $c_k < \rho$.  Since the set of points in
$k$ with height bounded by a constant is finite, we conclude that $H(\alpha) \le c_k$.  This proves Theorem \ref{thm1}.

We note that for the collection of algebraic number fields $k$ having an embedding into $\R$,
Ruppert's Question \ref{con1} has a positive answer with $B = 1$.  In particular, for this collection of number fields the 
constant $B$ is independent of the degree of $k$ over $\Q$.

\section{A general strategy}

In this section we consider an alternative argument using Minkowski's first theorem, but with a different choice of 
$\gamma = (\gamma_v)$.  In view of Theorem \ref{thm1} we are mainly interested in the case where $k$ has no real 
embedding, but the argument we develop here applies to all number fields $k$.
Let $P$ be a finite set of rational prime numbers.  Assume that for each prime $p$ in $P$ there exists a place $v$ of $k$ such that
\begin{equation}\label{gen30}
v|p\quad\text{and}\quad f_v(k/\Q) = 1,
\end{equation}
where $f_v(k/\Q)$ is the residue class degree of the place $v$ (or alternatively, the residue class degree of the 
associated prime ideal
\begin{equation*}\label{gen31}
\pp = \{\xi \in O_k: |\xi|_v < 1\},
\end{equation*}
where $O_k$ is the ring of algebraic integers in $k$.)
Then let $S$ be a set of non-archimedean places of $k$ selected so that $S$ contains exactly one place $v$ of
$k$ for each prime number $p$ in $P$, and the places $v$ in $S$ satisfy (\ref{gen30}).  Obviously we have $|P| = |S|$.

Recall that if $v$ is a finite place, then the positive integer $d_v(k/\Q)$ factors as 
\begin{equation}\label{gen39}
d_v(k/\Q) = e_v(k/\Q) f_v(k/\Q),
\end{equation}
where $e_v = e_v(k/\Q)$ is the index of ramification and $f_v = f_v(k/\Q)$ is the local residue class degree.  If $v|p$
then we write $\pi_v$ for an element that generates the unique maximal ideal in the integral domain $O_v$.  We find that
\begin{equation}\label{gen40}
\|\pi_v\|_v^{e_v} = p^{-1},\quad\text{and}\quad |\pi_v|_v = p^{-f_v/d}.
\end{equation} 
Next we select $\gamma = (\gamma_v)$ in $k_{\A}^{\times}$ so that
\begin{equation}\label{gen45}
\gamma_v = \begin{cases}  \pi_v^{-1} & \text{if $v$ belongs to $S$,}\\
                                   1 & \text{if $v$ does not belong to $S$.}\end{cases}
\end{equation}
Using (\ref{gen9}) and (\ref{gen40}) we find that 
\begin{equation}\label{gen47}
\beta\Big\{\prod_v \gamma_v O_v\Big\} = \Big\{2 c_k^{-1} \prod_{v \in S} |\pi_v^{-1}|_v\Big\}^d = \big\{2 c_k^{-1}\big\}^d \prod_{p \in P} p.
\end{equation}
We now assume that $P$ is selected so that
\begin{equation}\label{gen50}
c_k < \Big\{\prod_{p\in P} p\Big\}^{1/d}.
\end{equation}
By the adelic form of Minkowski's first theorem there exists a nonzero point $\alpha$ in the set
\begin{equation}\label{gen51}
k \cap \prod_v \gamma_v O_v.
\end{equation}

\begin{theorem}\label{thmgen2}  Let $P$ be a finite set of rational primes satisfying the above conditions, and in particular
satisfying the inequality {\rm (\ref{gen50})}.  Then for each nonzero point $\alpha$ contained in the set {\rm (\ref{gen51})}, we have
\begin{equation}\label{gen63}
\Q(\alpha) = k,\quad\text{and}\quad H(\alpha) \le \Big\{\prod_{p \in P} p\Big\}^{1/d}.
\end{equation}
\end{theorem}

\begin{proof}  A nonzero point $\alpha$ contained in the set (\ref{gen51}) satisfies the inequality
\begin{equation*}\label{gen52}
|\alpha|_v < 1
\end{equation*}
at each infinite place $v$ of $k$.  Hence it must satisfy
\begin{equation}\label{gen53}
1 < |\alpha|_w \le |\pi_w|_w^{-1}
\end{equation}
for at least one place $w$ from the set $S$.  Because the multiplicative value group of $|\ |_w$ on $k^{\times}$ is
given by
\begin{equation*}\label{gen54}
\big\{|\pi_w|_w^m: m \in \Z\big\}, 
\end{equation*}
the inequality (\ref{gen53}) implies that
\begin{equation}\label{gen55}
|\alpha|_w = |\pi_w|_w^{-1} = q^{1/d}, 
\end{equation}
where $q$ is the unique prime number in $P$ such that $w|q$.  Also, the height of $\alpha$ satisfies the bound
\begin{align}\label{gen56}
\begin{split}
H(\alpha) = \prod_v \max\{1, |\alpha|_v\}
          \le \prod_{v \in S} \max\{1, |\pi_v|_v^{-1}\}
          = \Big\{\prod_{p \in P} p\Big\}^{1/d}.
\end{split}
\end{align}
This verifies the inequality on the right of (\ref{gen63}).

Next we assume that $\Q(\alpha) = k^{\prime} \subseteq k$.  Let $u$ be a place of $k^{\prime}$ such
that $u|q$ and $w|u$.  Then the ramification indices satisfy the identity
\begin{equation}\label{gen57}
e_w(k/\Q) = e_w(k/k^{\prime}) e_u(k^{\prime}/\Q).
\end{equation} 
And we can write (\ref{gen55}) as
\begin{equation}\label{gen58}
\log \|\alpha\|_w = \frac{\log q}{e_w(k/\Q)}.
\end{equation}
As $\alpha$ belongs to $k^{\prime}$, we also get
\begin{equation}\label{gen59}
\log \|\alpha\|_w = \log \|\alpha\|_u = \frac{m \log q}{e_u(k^{\prime}/\Q)}
\end{equation}
for some positive integer $m$.  Combining (\ref{gen58}) and (\ref{gen59}) leads to the identity
\begin{equation}\label{gen60}
m e_w(k/\Q) = e_u(k^{\prime}/\Q).
\end{equation}
Then (\ref{gen57}) and (\ref{gen60}) imply that $m = 1$ and 
\begin{equation}\label{gen70}
e_w(k/k^{\prime}) = 1.
\end{equation}

Again because $\alpha$ belongs to the subfield $k^{\prime}$, the map $v\mapsto \|\alpha\|_v$ is constant on the
collection of places $v$ of $k$ such that $v|u$.  By our construction of $S$ there is only one place $v$ in $S$ such that
$v|q$ and $1 < |\alpha|_v$.  Hence the collection of places $v$ of $k$ such that $v|u$ consists of exactly the 
place $w$.  Using the hypothesis
\begin{equation}\label{gen71}
f_w(k/\Q) = f_w(k/k^{\prime}) f_u(k^{\prime}/\Q) = 1
\end{equation}
and (\ref{gen70}), this implies that
\begin{equation*}\label{gen72}
[k : k^{\prime}] = [k_w : k_u^{\prime}] = e_w(k/k^{\prime}) f_w(k/k^{\prime}) = 1.
\end{equation*} 
We have shown that $\Q(\alpha) = k^{\prime} = k$.
\end{proof}

\section{Application of Chebotarev's density theorem and GRH}

Let $L/K$ be a normal extension of algebraic number fields, and let $C$ denote a conjugacy class in the Galois
group $\Aut(L/K)$.  Let $\pp$ denote a prime ideal in the ring $O_K$ of algebraic integers in $K$.  If $\pp$
is unramified in $L$, we use the Artin symbol
\begin{equation*}\label{cheb1}
\biggl[\frac{L/K}{\pp}\biggr]
\end{equation*}
attached to $\pp$ to denote the conjugacy class of Frobenius automorphisms that correspond to prime ideals $\pP$ in $O_L$
such that $\pP|\pp$.  Then for $2 \le x$ we write $\pi_C(x; L/K)$ for the cardinality of the set
of prime ideals $\pp$ in $O_K$ such that $\pp$ is unramified in $L$, 
\begin{equation*}\label{cheb2}
\biggl[\frac{L/K}{\pp}\biggr] = C, 
\end{equation*}
and
\begin{equation*}\label{cheb3}
\Norm_{K/\Q}(\pp) \le x.
\end{equation*}
In its most basic form the Chebotarev density theorem (see \cite{cheb1926}) asserts that
\begin{equation}\label{cheb4}
\lim_{x\rightarrow \infty} \frac{\pi_C(x; L/K)}{\Li(x)} = \frac{|C|}{[L:K]},
\end{equation}
where $|C|$ is the cardinality of the conjugacy class $C$, and
\begin{equation*}\label{cheb5}
\Li(x) = \int_2^x \frac{1}{\log t}\ \dt
\end{equation*}
is the logarithmic integral.  

For our purposes it is useful to have an explicit estimate for the rate of 
convergence in (\ref{cheb4}).  And it is important that the estimate apply for relatively small values of the parameter $x$.
Such an explicit, but conditional, estimate is given by a well known result of Lagarias and Odlyzko \cite[Theorem 1.1]{lo1977}, which 
we now describe.  Let $\zeta_L(s)$ denote the Dedekind zeta-function associated to the
number field $L$, where $s = \sigma + it$.  We assume that $\zeta_L(s)$ satisfies the Generalized Riemann Hypothesis.  Then 
\cite[Theorem 1.1]{lo1977} implies that there exists an absolute and effectively computable constant  $c_1 \ge 1$,
such that, if $2 \le x$ then
\begin{equation}\label{cheb8}
\biggl|\pi_C(x; L/K) - \frac{|C|}{[L:K]} \Li(x)\biggr| \le c_1x^{\h}\bigl(\log |\Delta_L| + [L:\Q] \log x\bigr).
\end{equation}

Let $k$ be an algebraic number field of degree $d$ over $\Q$.  We apply the estimate (\ref{cheb8}) with $L$ equal to the 
Galois closure of the extension $k/\Q$, and with $K = \Q$.  Using the conjugacy class $C = \{1\}$, this will establish the existence 
of a rational prime number $p$ that can be used to satisfy the hypotheses of Theorem \ref{thmgen2}.

\begin{lemma}\label{cheblem1}  Let $k$ be an algebraic number field of degree $d \ge 2$ over $\Q$, and let $l$ be the Galois closure 
of the extension $k/\Q$.  Assume that the Dedekind zeta-function $\zeta_l(s)$ associated to $l$ satisfies the Generalized Riemann 
Hypothesis, and let $C$ be a conjugacy class in the Galois group $\Aut(l/\Q)$.  If
\begin{equation}\label{cheb20}
(15)^{20} c_1^{20} (d!)^{60} \le |\Delta_k|,
\end{equation} 
then we have
\begin{equation}\label{cheb21}
1 \le \pi_C\bigl(2 |\Delta_k|^{\h}; l/\Q\bigr) - \pi_C\bigl(|\Delta_k|^{\h}; l/\Q\bigr).
\end{equation}
\end{lemma}

\begin{proof}  As $l$ is the Galois closure of $k/\Q$ we find that $[l:\Q] \le d!$, and
\begin{equation}\label{cheb22}
\log |\Delta_l| \le 2(d!)^2 \log |\Delta_k|.
\end{equation}
The inequality (\ref{cheb22}) follows because a rational prime that ramifies in $l$ must also ramify in $k$.  Then by 
\cite[Theorem B.2.12.]{BG} the order to which a rational prime divides $\Delta_l$ is bounded from above by $2[l:\Q]^2$.  We
apply these observations to the inequality (\ref{cheb8}) with $L = l$ and $K = \Q$.  It follows that for $2 \le x$ we have
\begin{equation}\label{cheb23}
\biggl|\pi_C(x; l/\Q) - \frac{|C|}{[l:\Q]} \Li(x)\biggr| \le 2 c_1(d!)^2 x^{\h}\bigl(\log |\Delta_k| + \log x\bigr).
\end{equation}

If the nonnegative integer on the right of (\ref{cheb21}) is zero, then (\ref{cheb23}) implies that
\begin{equation}\label{cheb24}
\Li\bigl(2|\Delta_k|^{\h}\bigr) - \Li\bigl(|\Delta_k|^{\h}\bigr) \le 10 c_1 (d!)^3 |\Delta_k|^{\frac14} \log |\Delta_k|.
\end{equation}
Therefore we get
\begin{equation*}\label{cheb24.5}
\frac{2|\Delta_k|^{\h}}{3\log |\Delta_k|} \le \frac{|\Delta_k|^{\h}}{\log 2|\Delta_k|^{\h}}
                                          \le \int_{|\Delta_k|^{\h}}^{2|\Delta_k|^{\h}} \frac{1}{\log t}\ \dt
                                          \le 10 c_1 (d!)^3 |\Delta_k|^{\frac14} \log |\Delta_k|,
\end{equation*}
and then
\begin{equation}\label{cheb25}
|\Delta_k|^{\h} \le 15 c_1 (d!)^3 |\Delta_k|^{\frac14} \bigl(\log |\Delta_k|\bigr)^2.
\end{equation}
It is now obvious that (\ref{cheb25}) is false if $|\Delta_k|$ is sufficiently large.  An elementary calculation shows
that (\ref{cheb25}) is false if $|\Delta_k|$ satisfies the inequality (\ref{cheb20}).  Therefore (\ref{cheb20}) implies that 
the nonnegative integer on the right of (\ref{cheb21}) is positive.
\end{proof}

\section{Proof of Theorem \ref{thm2}}

We assume that the number field $k$ satisfies the hypotheses of Theorem \ref{thm2}, and we also assume that
\begin{equation*}\label{cheb30}
(15)^{20} c_1^{20} (d!)^{60} \le |\Delta_k|.
\end{equation*} 
We apply Lemma \ref{cheblem1} with the conjugacy class $C = \{1\}$.  It follows that there exists a rational prime number $p$
such that $\langle p\rangle$ splits completely in $O_l$ and
\begin{equation*}\label{cheb31}
|\Delta_k|^{\h} < p \le 2 |\Delta_k|^{\h}.
\end{equation*}
Then $\langle p \rangle$ splits completely in $O_k$, and therefore the residue class degrees of all prime ideal factors
of $\langle p \rangle$ in $O_k$ are equal to $1$.  As each prime ideal factor corresponds to a non-archimedean place of $k$,
we find that the hypotheses of Theorem \ref{thmgen2} are satisfied with $P = \{p\}$.  We
conclude that there exists an element $\alpha$ in $k$ such that $k = \Q(\alpha)$, and
\begin{equation}\label{cheb32}
H(\alpha) \le p^{1/d} \le 2 |\Delta_k|^{1/2d}.
\end{equation}

By Hermite's theorem there are only finitely many algebraic number fields $k$ having degree $d$ and satisfying the inequality
\begin{equation}\label{cheb33}
\Delta_k| < (15)^{20} c_1^{20} (d!)^{60}.
\end{equation}
As these can be effectively determined, there exists an effectively computable positive number $B = B(d) \ge 2$ such 
that the conclusion (\ref{gen90}) holds for each field $k$ having degree $d$ and satisfying (\ref{cheb33}).  In 
view of (\ref{cheb32}), the conclusion (\ref{gen90}) holds for all fields $k$ of degree $d$.

\section{The field $\Q\bigl(\sqrt{-163}\bigr)$}

Theorem \ref{thmgen2} can be used to establish the existence of a generator $\alpha$ of
$k/\Q$ with relatively small height.  The bound obviously depends on identifying a finite set $P$ of rational prime
numbers that satisfies the hypotheses of that result.  It may be of interest to observe that there are nontrivial examples
where the bound obtained by applying Theorem \ref{thmgen2} is sharp.  In particular this is so for the imaginary
quadratic field $\Q\bigl(\sqrt{-163}\bigr)$.  

\begin{lemma}\label{lemquad1}  Let $d \le -1$ be a square free integer, let
\begin{equation*}\label{quad1}
f(x) = ax^2 + bx + c
\end{equation*}
be a polynomial in $\Z[x]$ with
\begin{equation*}\label{quad2}
1 \le a,\quad 1\le c,\quad \gcd(a, b, c) = 1,\quad\text{and}\quad b^2 - 4ac = de^2, 
\end{equation*}
where $e$ is a nonzero integer.  If $\alpha$ is a root of $f$, then $\Q(\alpha) = \Q(\sqrt{d})$ and
\begin{equation}\label{quad3}
H(\alpha) = \max\{a, c\}^{\h}.
\end{equation}
\end{lemma}

\begin{proof}  That $\Q(\alpha) = \Q(\sqrt{d})$ is obvious.  For the remainder of the proof
we work in $\Q(\sqrt{d})$ embedded in $\C$.  Then complex conjugation is the unique nontrivial automorphism of $\Q(\sqrt{d})$,
and the distinct roots of $f$ are $\alpha$ and $\bbalpha$.
Hence the Mahler measure of $f$ is
\begin{equation*}\label{quad4}
M(f) = a \max\{1, |\alpha|\} \max\{1, |\bbalpha|\} = a \max\{1,  \alpha\bbalpha\} = \max\{a, c\}.
\end{equation*}
Because $f$ is the unique irreducible polynomial in $\Z[x]$ with positive leading coefficient and a root at $\alpha$,
the Mahler measure is also given by
\begin{equation*}\label{quad5}
M(f) = H(\alpha) H(\bbalpha) = H(\alpha)^2. 
\end{equation*}
The result follows by combining these identities.
\end{proof}

A rational prime $p$ will satisfy the hypotheses (\ref{gen30}) for the field $\Q(\sqrt{-163})$ if and only if either $p$ is
odd and $-163$ is a quadratic residue  modulo $p$, or $p = 163$.  The smallest odd prime number $p$ such that $-163$ is
a quadratic residue modulo $p$ is $41$.  We have $\Delta_k = -163$ and therefore
\begin{equation*}\label{quad71}
c_k = 2.850 \cdots < \sqrt{41} = 6.403 \cdots.
\end{equation*}
We conclude that Theorem \ref{thmgen2} applies with $P = \{41\}$, and asserts that $\Q(\sqrt{-163})$ has a generator $\alpha$
such that
\begin{equation*}\label{quad72}
H(\alpha) \le \sqrt{41}.
\end{equation*}

Now let $\alpha^{\prime}$ be a root of the polynomial $x^2 + x + 41$, which has discriminant $-163$.  Clearly $\alpha^{\prime}$ 
generates the imaginary quadratic field $\Q\bigl(\sqrt{-163}\bigr)$.  Then by Lemma \ref{lemquad1} we have 
\begin{equation*}\label{quad73}
H\bigl(\alpha^{\prime}\bigr) = \sqrt{41}.
\end{equation*}
The set of polynomials $f(x) = ax^2 + bx + c$ in $\Z[x]$ such that
\begin{equation*}\label{quad74}
1 \le a,\quad 1\le c,\quad \gcd(a, b, c) = 1,\quad\text{and}\quad b^2 - 4ac = (-163)e^2, 
\end{equation*}
where $e$ is a nonzero integer, and
\begin{equation}\label{quad75}
\max\{a, c\} \le 41,
\end{equation}
is obviously finite.  It is then a simple matter, using Lemma \ref{lemquad1}, to check that the the field 
$\Q\bigl(\sqrt{-163}\bigr)$ does not have a generator with height strictly
smaller than $\sqrt{41}$.  Therefore the inequality obtained in Theorem \ref{thmgen2} for this field is sharp.

\bibliographystyle{amsplain}
\bibliography{literature}

\end{document}